\documentclass[11pt]{amsart}
\usepackage{graphicx}
\usepackage{amsthm,amsmath,amsfonts,amssymb,natbib}

\RequirePackage[colorlinks,citecolor=blue,urlcolor=blue]{hyperref}

\newtheorem{theorem}{Theorem}[section]
\newtheorem{corollary}[theorem]{Corollary}
\newtheorem{lemma}[theorem]{Lemma}
\newtheorem{assumption}[theorem]{Assumption}
\newtheorem{proposition}[theorem]{Proposition}
\newtheorem{definition}[theorem]{Definition}
\numberwithin{equation}{section} \theoremstyle{remark}
\newtheorem{remark}[theorem]{Remark}

\begin{document}

\title[Gibbs random fields on unbounded degree graphs]
{Gibbs random fields with unbounded spins on unbounded degree graphs}%
\author{Yuri Kondratiev}%
\address{Fakult\"at f\"ur Mathematik\\ Universit\"at Bielefeld, D 33615, Germany}%
\email{kondrat@math.uni-bielefeld.de}
\author{Yuri Kozitsky}
\address{Instytut Matematyki,\\ Uniwersytet Marii Curie-Sk{\l}odowskiej\\ 20-031 Lublin,
Poland} \email{jkozi@hektor.umcs.lublin.pl}
\author{Tanja Pasurek}
\address{Fakult\"at f\"ur Mathematik\\ Universit\"at Bielefeld, D 33615, Germany}
\email{tpasurek@@math.uni-bielefeld.de}%

\thanks{This work was financially supported by the DFG through SFB 701: ``Spektrale Strukturen
und Topologische Methoden in der Mathematik"  and through the
research project 436 POL 125/0-1. Yuri Kozitsky was also supported
by TODEQ MTKD-CT-2005-030042.}%

\today
\begin{abstract}
Gibbs random fields corresponding to systems of real-valued spins
(e.g. systems of interacting anharmonic oscillators) indexed by the
vertices of unbounded degree graphs with a certain summability
property are constructed. It is  proven that the set of tempered
Gibbs random fields is non-void and weakly compact, and that they
obey uniform exponential integrability estimates. In the second part
of the paper, a class of graphs is described in which the mentioned
summability is obtained as a consequence of a property, by virtue of
which  vertices of large degree are located at large distances from
each other. The latter is a stronger version of a metric property,
introduced in [Bassalygo, L. A. and Dobrushin, R. L. (1986).
\textrm{Uniqueness of a Gibbs field with a random potential---an
elementary approach. }\textit{ Theory Probab. Appl.} {\bf 31}
572--589].

\end{abstract}
\maketitle
\section{Introduction and paper overview} \label{ISsec}

\subsection{Introduction}

Gibbs random fields on a discrete metric space (e.g. on a graph) can
be viewed as collections of dependent random variables, usually
called {\it spins}, indexed by the elements of this space. Their
joint probability laws are defined by the families of local
conditional distributions constructed by means of interaction
potentials.
 We quote the monographs
\cite{[Ge],[GeH]} as standard sources in the theory of such fields.
Each spin takes values in the corresponding {\it single-spin} space,
say $\mathcal{X}_x$. Most of Gibbs random fields constructed on
general graphs correspond to models with finite single-spin spaces.
Perhaps, the most known example is the Ising model where
$\mathcal{X}_x = \{-1, 1\}$ for all $x$. By the compactness of
$\mathcal{X}_x$, such Gibbs fields exist for arbitrary graphs, see
\cite{[GeH],[Ha],[Ha1],[Ly1],[Ly2],[Weitz]}. Their properties are
closely related to those of random walks or corresponding
percolation models, see e.g. \cite{[Ha],[Ly1],[Ly2]}. The
development of the theory  of Gibbs random fields with unbounded
spins, started in the late seventies in the pioneering works
\cite{[LeP],[COP]},  was strongly motivated  by physical
applications, especially, in Euclidean quantum field theory, see
e.g. \cite{[Simon]}. Since that time, such random fields were
extensively studied, see e.g. the bibliographical notes in
\cite{[P]}. However, the results obtained in all these works were
restricted to the case where the underlying metric space is a simple
cubic lattice $\mathbb{Z}^d$. In \cite{[KP],[P]}, the theory of
Gibbs random fields was extended to unbounded spin systems living on
more general discrete metric spaces, including graphs of bounded
degree. In this context, we mention also the paper \cite{[Hattori]}
where a Gaussian field on a bounded degree graph was studied.

In the present paper, we construct Gibbs random fields with
unbounded spins ($\mathcal{X}_x = \mathbb{R}$ for all $x$) on
unbounded degree graphs of  certain kind and analyze the role played
here by the geometry of the graph. In doing so, we are motivated by
the following reasons:
\begin{itemize}
\item Random fields on Riemannian manifolds, especially those associated
with the corresponding Laplace-Beltrami operators, cf. \cite{[DeA]},
can be approximated by their discrete versions living on appropriate
graphs \cite{[Fujiwara]}. This includes also the case of quantum
fields in curved spacetime, see \cite{[Zegar],[Wald]}.
\item As the degree of the graph can be related to such a property of the corresponding manifold as curvature,
the use of unbounded degree graphs essentially extends the class of
manifolds that can be approximated in the above sense.
\item Another application can be the description
of systems of interacting oscillators located at vertices of an
infinite graph -- the so called oscillating networks, see Section 14
in \cite{[BC]}. We refer also to the survey \cite{[BCV]},  where
other relevant physical models can be found.
\end{itemize}
The results of the paper are: (a) constructing Gibbs random fields;
(b) deriving exponential integrability estimates and support
properties for such fields; (c) presenting a concrete family of
unbounded degree graphs, which can serve as underlying graphs for
our model. In achieving (a) and (b), we used  a modification of the
technique developed in \cite{[KP],[P]}. In constructing the family
of (c) we were inspired by some aspects of \cite{[BD]}. To the best
of our knowledge, the present study is the first attempt in such a
direction. We plan to continue investigating the model introduced
here in forthcoming papers. In particular, we are going to study the
problem of uniqueness of Gibbs random fields, as well as the
ergodicity properties of the corresponding stochastic dynamics.

\subsection{The paper overview}
The model we deal with in this paper is the triple $({\sf G}, W,
V)$, where ${\sf G} = ({\sf V}, {\sf E})$ is a graph,
$W:\mathbb{R}\times \mathbb{R}\rightarrow \mathbb{R}$ and
$V:\mathbb{R}\rightarrow \mathbb{R}$ are continuous functions
(potentials). The properties of the triple $({\sf G}, W, V)$ are
specified below in Assumption \ref{ASS}, see also (\ref{A150}). This
triple determines the heuristic Hamiltonian
\begin{equation}
 \label{hh}
H (\omega) = \sum_{\langle x, y \rangle} W(\omega(x), \omega(y)) +
\sum_{x} V(\omega (x)),
\end{equation}
where the first (resp. second) sum is taken over all edges (resp.
vertices) of the graph. For this model,  Gibbs random fields are
defined as probability measures on the configuration space
$\mathit{\Omega} = \mathbb{R}^{\sf V}$. In contrast to the case of
bounded spins, it is unrealistic to describe  all
 Gibbs measures of an unbounded spin
system without assuming a priori  any of its properties. Thus, among
all Gibbs measures corresponding to (\ref{hh}) we distinguish those
that have a prescribed support property, i.e., such that $\mu
(\mathit{\Omega}^{\rm t})=1$ for an a priori chosen proper subset
$\mathit{\Omega}^{\rm t} \subset \mathit{\Omega}$. These measures
are called {\it tempered}. In Theorem \ref{T2}, we show that the set
of tempered Gibbs measures $\mathcal{G}^{\rm t}$ is non-void and
weakly compact. Here we also show that each $\mu\in\mathcal{G}^{\rm
t}$ obeys important integrability estimates, the same for all such
measures. In Theorems \ref{T22} and \ref{T3}, these results are
extended in the following directions: (a) we allow the potential $W$
to be super-quadratic, see (\ref{A150}); (b) we consider a scale of
sets of tempered Gibbs measures, which clarifies connections between
the graph geometry and the properties of such measures. These our
results are valid for any graph possessing the summability specified
in Assumption \ref{ASS}. To provide a nontrivial example of
unbounded degree graphs with this property, in the second part of
the paper we introduce a new class of such graphs, which we believe
is interesting in its own right. This class is characterized by the
following property, cf. (\ref{Qq0}) and (\ref{I}). For vertices $x$
and $y$, such that their degrees, $n(x)$ and $n(y)$, exceed some
threshold value, the path distance is supposed to obey the
`repulsion' condition
 \begin{equation} \label{rep}
\rho (x,y) \geq \phi\left[ \max\{n(x), n(y) \}\right],
\end{equation} where $\phi$ is a given increasing function. In such
graphs, every vertex $x$ has the property that
\[
\sup_{y: \ \rho(x,y) \leq N} n(y) \leq \phi^{-1} (2N),
\]
whenever $N$
 exceeds some $N_x$,
specific for this $x$. By means of this property, for $\phi(b) =
\upsilon \log b [\log\log b]^{1+\varepsilon}$, $\upsilon,
\varepsilon >0$, we obtain the estimate
\[
\sum_{y: \ \rho(x,y) = N} [n(y)]^{1+\theta} \leq \exp(a N),
\]
which holds for any $\theta>0$ and an appropriate $a>0$, whenever
$N\geq N_x$.  In Theorem \ref{Itm}, we show that the latter estimate
implies the required summability (\ref{ass}).

The rest of the paper is organized as follows. In the first part,
the emphasis is put on the probabilistic stuff, whereas the second
part -- Section \ref{RGsec} -- is devoted to the graph-theoretical
aspects of the problem. In Section \ref{Qsec}, we specify the class
of models by imposing conditions on the graph and on the potentials.
The only essential condition imposed on ${\sf G}$ is the summability
(\ref{ass}). The potentials are supposed to obey quite standard
stability requirements, plus continuity. We note, however, that the
stability condition (\ref{A16}) is a bit stronger than the one with
$q=2$, typical for graphs of bounded degree. In view of this fact,
the Gaussian case is not covered by our theory. Thereafter, we put
forward Theorem \ref{T2}. In Section \ref{PEsec}, we present
Theorems \ref{T22} and \ref{T3}. The proof of the latter theorem
follows from the estimates obtained in Theorem \ref{T22}. The proof
of Theorem \ref{T22}, which is the main technical component of the
first part of the paper, is given in Section \ref{PROsec}.
 It is
preceded by a number of lemmas, in which we elaborate the
corresponding tools. The key element here is Lemma \ref{A2lm} the
proof of which crucially employs the summability (\ref{ass}). In
Section \ref{RGsec}, we introduce and describe the class of graphs
with the property (\ref{rep}), which by Theorem \ref{Itm} can serve
as underlying graphs for our model.

\section{The setup and the basic theorem}

\label{Qsec}

\subsection{The model}
The underlying graph ${\sf G} = ({\sf V}, {\sf E})$ of the model
(\ref{hh}) is supposed to be  undirected and countable. Two adjacent
vertices $x,y\in {\sf V}$ are also called neighbors. In this case,
we write $x\sim y$ and $\langle x,y\rangle \in {\sf E}$. The {\it
degree} of $x\in {\sf V}$, denoted by $n(x)$, is the cardinality of
the neighborhood of $x$, that is, of the set $\{ y \ | \ y\sim x\}$.
We use the shorthand
\begin{equation*}
\sum_{x} = \sum_{x\in {\sf V}} \qquad \ \sup_x = \sup_{x\in {\sf V}}
\qquad \ \sum_{y\sim x} = \sum_{y\in {\sf V} : \ y\sim x}.
\end{equation*}
The graph  is assumed to be {\it locally finite},  which means that
$n(x) \in \mathbb{N}$ for any $x$. At the same time, we assume that
$\sup_{x} n(x) = +\infty$, which is reflected in the title of the
paper. Of course, our results are trivially valid for bounded degree
graphs.

A sequence $\vartheta=\{x_0, x_1, \dots , x_n \}$, such that $x_k
\sim x_{k+1}$ for all $k=0, \dots , n-1$, is called a {\it path}.
Herein, some of the vertices may be repeated. The path connects its
{\it endpoints} $x_0$ and $x_n$; it {\it leaves} the vertices $x_0,
\dots , x_{n-1}$ and {\it enters} $x_1, \dots , x_n$. The number of
left vertices, denoted by $\|\vartheta\|$, is called the {\it
length} of the path. For $x,y\in {\sf V}$, by $\vartheta (x,y)$ we
denote a path, whose endpoints are $x$ and $y$. We assume that ${\sf
G}$ is connected, which means that there exists a path $\vartheta
(x,y)$ for every  $x$ and $y$. The path distance $\rho(x,y)$ is set
to be the length of the shortest $\vartheta(x,y)$. It is a metric on
${\sf G}$ by means of which, for a certain $o\in{\sf V}$ and $\alpha
>0$, we define
\begin{equation}
  \label{A11}
 w_\alpha (x) = \exp[- \alpha \rho(o,x)],
  \quad \ x\in {\sf V}.
\end{equation}
For $\theta >0$, we also set\footnote{In mathematical chemistry, the
sum of terms $[n(x) n(y)]^\theta$ taken over the edges $\langle x,y
\rangle$ of a finite tree is known under the name {\it generalized
Randi\'c} or {\it connectivity}  index, see e.g. \cite{[CG]}.}
\begin{equation}
  \label{Pq}
m_\theta (x) = \sum_{y\sim x} [n(x) n(y)]^\theta, \quad \ x\in {\sf
V}.
\end{equation}
The remaining  properties of the model are summarized in
\begin{assumption} \label{ASS}
The triple $({\sf G}, W, V)$ is subject to the following conditions:
 \vskip.1cm
\begin{tabular}{ll}
(i) \ &the graph ${\sf G}$ is such that, for some positive $\alpha$
and $\theta$,
\end{tabular}
\begin{equation}
  \label{ass}
\mathit{\Theta} (\alpha , \theta) \ \stackrel{\rm def}{=}\  \sum_{x}
m_\theta (x) w_\alpha (x) < \infty;
\end{equation}
\begin{tabular}{ll}
 \ (ii) \ &the function $W$ is continuous, symmetric, and such that
\end{tabular}
\begin{equation}
  \label{A15}
|W(u, v)| \leq [I_W + J_W (u^2 + v^2)]/2,
\end{equation}
\begin{tabular}{ll}
 &for some $I_W, J_W >0$ and all $u, v \in \mathbb{R}$;\\[.2cm]
\ (iii) \ &the function $V$ is continuous and such that, for all
$u\in \mathbb{R}$,
\end{tabular}
\begin{equation}
  \label{A16}
 V(u) \geq a_V |u|^q - c_V,
\end{equation}
\begin{tabular}{ll}
\qquad \ \  &for some $a_V, c_V >0$ and  $q> 2 +2/\theta$, with
$\theta$ being the same\\[.2cm] &as in (i).
\end{tabular}
\end{assumption}

\subsection{The basic result}
Following the standard DLR route, see \cite{[Ge]},
 the Gibbs random fields
for our model are defined as  probability measures on the measurable
space $(\mathit{\Omega},\mathcal{B}(\mathit{\Omega}))$. Here
$\mathit{\Omega} = \mathbb{R}^{\sf V}$ is the configuration space,
 equipped with the product topology and with the corresponding
Borel $\sigma$-field $\mathcal{B}(\mathit{\Omega})$. By
$\mathcal{P}(\mathit{\Omega})$ we denote the space of all
probability measures on
$(\mathit{\Omega},\mathcal{B}(\mathit{\Omega}))$, which is equipped
with the weak topology determined by bounded continuous functions
$f:\mathit{\Omega}\rightarrow \mathbb{R}$. By $C_{\rm
b}(\mathit{\Omega})$ we denote the set of all such functions.

In the sequel, by writing $\mathit{\Lambda} \Subset {\sf V}$ we mean
that $\mathit{\Lambda}$ is a finite and non-void set of vertices. A
property related to such a subset is called {\it local}. As usual,
$\mathit{\Lambda}^c = {\sf V}\setminus\mathit{\Lambda}$ stands for
the complement of  $\mathit{\Lambda}\subset {\sf V}$.

For $\mathit{\Lambda}\Subset {\sf V}$ and $\omega \in
\mathit{\Omega}$, by $\omega_{\mathit{\Lambda}}$ we denote the
restriction of $\omega$ to $\mathit{\Lambda}$, and use the
decomposition $\omega = \omega_{\mathit{\Lambda}} \times
\omega_{\mathit{\Lambda}^c}$. Then for such $\mathit{\Lambda}$ and a
fixed $\xi \in \mathit{\Omega}$, the relative local Hamiltonian is
set to be
\begin{eqnarray} \label{ass1}
H_{\mathit{\Lambda}} (\omega_{\mathit{\Lambda}}|\xi) & = &
\sum_{\langle x , y\rangle : \ x,y\in
\mathit{\Lambda}} W(\omega(x) , \omega(y)) \\[.2cm] & + &\sum_{\langle x , y\rangle : \ x\in
\mathit{\Lambda}, \ y\in \mathit{\Lambda^c}} W(\omega(x) , \xi(y)) +
\sum_{x\in \mathit{\Lambda}}V(\omega (x)). \nonumber
\end{eqnarray}
 Thereby, for
$\mathit{\Lambda}\Subset {\sf V}$, $\xi\in \mathit{\Omega}$, and
$A\in \mathcal{B}(\mathit{\Omega})$, we define
\begin{eqnarray} \label{ass2}
  \pi_{\mathit{\Lambda}} (A|\xi)
  & = & \frac{1}{Z_\mathit{\Lambda}
  (\xi)}\int_{{\mathbb{R}}^{|\mathit{\Lambda}|}} \mathbb{I}_A (\omega_{\mathit{\Lambda}}\times \xi_{\mathit{\Lambda}^c})
\exp\left[ -H_{\mathit{\Lambda}}
(\omega_{\mathit{\Lambda}}|\xi)\right] {\rm d}
\omega_{\mathit{\Lambda}}, 
\end{eqnarray}
where $\mathbb{I}_A$ is the indicator function, ${\rm d}
\omega_\mathit{\Lambda}$ is the corresponding Lebesgue measure on
the Euclidean space $\mathit{\Omega}_\mathit{\Lambda} \stackrel{\rm
def}{=} \mathbb{R}^{|\mathit{\Lambda}|}$, and $Z_\mathit{\Lambda}
(\xi)$ is a normalizing factor. Hence, each $\pi_{\mathit{\Lambda}}
(\cdot |\xi)\in \mathcal{P}(\mathit{\Omega})$. The family
$\{\pi_{\mathit{\Lambda}}\}_{\mathit{\Lambda}\Subset {\sf V}}$ is
called the {\it local Gibbs specification} for the model we
consider. Directly from the definition (\ref{ass2}), one makes sure
that this family is consistent in the following sense:
\begin{equation}
  \label{A19a}
\int_{\mathit{\Omega}} \pi_{\mathit{\Delta}} (A|\omega)
\pi_{\mathit{\Lambda}} ({\rm d}\omega|\xi) = \pi_{\mathit{\Lambda}}
(A|\xi),
\end{equation}
which holds for all $A\in \mathcal{B}(\mathit{\Omega})$, all
$\mathit{\Delta}\subset \mathit{\Lambda}$, and all $
\mathit{\Lambda} \Subset {\sf V}$.
\begin{definition}\label{assdf}
A measure $\mu\in \mathcal{P}(\mathit{\Omega})$ is said to be a
Gibbs random field corresponding to the Hamiltonian (\ref{hh}) if it
solves the following (DLR) equation
\begin{equation}
  \label{dlr}
  \mu(A) = \int_{\mathit{\Omega}} \pi_{\mathit{\Lambda}} (A|\omega) \mu({\rm d}\omega),
\end{equation}
for all $A\in \mathcal{B}(\mathit{\Omega})$ and $\mathit{\Lambda}
\Subset {\sf V}$.
\end{definition}
An equivalent version of (\ref{dlr}) is the following equation
\begin{equation}
  \label{dlr1}
\mu(f) = \int_{\mathit{\Omega}} \pi_{\mathit{\Lambda}} (f|\omega)
\mu({\rm d}\omega),
\end{equation}
which ought to hold for all $f\in C_{\rm b}(\mathit{\Omega})$  and
$\mathit{\Lambda} \Subset {\sf V}$. Here, for such $f$ and $\mu \in
\mathcal{P}(\mathit{\Omega})$, we use the notation
\[
\mu(f) = \int_{\mathit{\Omega}} f(\omega) \mu({\rm d}\omega).
\]
 Let
$\mathcal{G}$ stand for the set of all solutions of (\ref{dlr}). As
is typical for unbounded spin systems, it is far from being obvious
whether $\mathcal{G}$ is non-void. But if it is the case, the
description of properties possessed by all the elements of
$\mathcal{G}$ is rather unrealistic. Thus, one constructs and
studies a subset of $\mathcal{G}$, consisting of the measures
possessing a prescribed (support) property. Such measures are called
{\it tempered}.

For positive $p$ and $\alpha$, we set
\begin{equation}
  \label{A12}
\|\omega\|_{p,\alpha} = \left[ \sum_{x\in {\sf V} }|\omega(x)|^p
w_\alpha (x)\right]^{1/p},
\end{equation}
where the weights $w_\alpha$ are defined in (\ref{ass}).
 Then
\begin{equation*} 
 L^p ({\sf V}, w_\alpha) =\{\omega\in
\mathbb{R}^{\sf V} \ | \ \|\omega\|_{p,\alpha}<\infty\},
\end{equation*}
is a Banach space. For $\theta$ and $q$ being as in (\ref{ass}) and
in (\ref{A16}), respectively,  we fix
\begin{equation}
  \label{p}
p = 2 + 2/\theta  < q.
\end{equation}
For this $p$, the set of {\it tempered} configurations is set to be
\begin{equation}
  \label{FEB}
 \mathit{\Omega}^{\rm t} = L^p ({\sf V}, w_\alpha),
\end{equation}
where $\alpha$ is as in (\ref{ass}).
 Clearly,
$\mathit{\Omega}^{\rm t}\in \mathcal{B}(\mathit{\Omega})$; hence,
one can define
\begin{equation}
 \label{t1}
\mathcal{G}^{\rm t} = \{\mu\in\mathcal{G} \ | \
\mu(\mathit{\Omega}^{\rm t}) =1 \}.
\end{equation}
\begin{theorem}[{\sc Basic}]
  \label{T2}
The set $\mathcal{G}^{\rm t}$ is non-void and weakly compact. For
every $\lambda >0$ and $x\in {\sf V}$, there exists a positive
constant $C(\lambda,x)$,  such that, for  all $\mu \in
\mathcal{G}^{\rm t}$,
\begin{equation}
  \label{A21}
 \int_{\mathit{\Omega}} \exp \left( \lambda |\omega(x)|^{p} \right)\mu({\rm d}\omega) \leq C(\lambda,x).
\end{equation}
Furthermore, for every $\lambda >0$, there exists a positive
constant $C(\lambda)$,  such that, for  all $\mu \in
\mathcal{G}^{\rm t}$,
\begin{equation}
  \label{A21A}
\int_{\mathit{\Omega}} \exp \left( \lambda \|\omega\|_{p,\alpha}^{p}
\right)\mu({\rm d}\omega) \leq C(\lambda).
\end{equation}
Herein, $\alpha$  and $p$ are the same as in (\ref{FEB}).
\end{theorem}

\section{Extensions} \label{PEsec}

\subsection{More on temperedness} In this subsection,
Theorem \ref{T2} is extended in the following directions: (a) we
allow a super-quadratic growth of the potential $W$, cf.
(\ref{A15}); (b) we construct a scale of sets of tempered Gibbs
fields, the elements of which obey integrability estimates, stronger
than (\ref{A21}) and (\ref{A21A}).

In what follows, instead of (\ref{A15}) we assume
\begin{equation}
  \label{A150}
|W(u, v)| \leq [I_W + J_W (|u|^r + |v|^r)]/2,
\end{equation}
for some $r>0$. The potential $V$ is assumed to obey (\ref{A16})
with $q> r + r/\theta$, where $\theta$ is  as in (\ref{ass}). The
graph ${\sf G}$ is supposed to be the same as in Assumption
\ref{ASS}. The scale of tempered Gibbs fields which we are going to
construct will be indexed by $\alpha$ and $p$. First, we set
\begin{equation} \label{A12B}
\underline{\alpha} = \inf\{ \alpha \ | \ \mathit{\Theta}(\alpha,
\theta) < \infty\},
\end{equation}
and let $\overline{\alpha}> \underline{\alpha}$ be such that
(\ref{ass}) holds for all $\alpha \in (\underline{\alpha},
\overline{\alpha}]$. Next, we define
\begin{equation}
  \label{A151}
 p_0 {=}  r + r/\theta.
\end{equation}
For $\alpha', \alpha \in (\underline{\alpha}, \overline{\alpha}]$
and $p',p\in [p_0, q)$, by (\ref{A12}) we have
\begin{equation} \label{B26}
L^{p'}({\sf V}, w_{\alpha'}) \hookrightarrow L^{p}({\sf V},
w_{\alpha}), \quad \ \
 {\rm whenever} \ \ \ \alpha ' \leq \alpha \ \ \ {\rm and} \ \ p'\geq p.
\end{equation}
Notably, the above embedding is compact. Then, for $ \alpha \in
(\underline{\alpha}, \overline{\alpha}]$ and $p\in [p_0, q)$, we
set, cf. (\ref{t1}),
\begin{equation}
  \label{A152}
\mathcal{G}_{p,\alpha} = \{\mu\in\mathcal{G} \ | \ \mu\left[
L^{p}({\sf V}, w_{\alpha})\right] =1 \}.
\end{equation}
Clearly
\begin{equation}
  \label{B25}
\mathcal{G}_{p', \alpha'} \subset \mathcal{G}_{p,\alpha}, \quad \ \
 {\rm whenever} \ \ \ \alpha ' \leq \alpha \ \ \ {\rm and} \ \ p'\geq p.
\end{equation}
The following statement is an extended version of Theorem \ref{T2}.
\begin{theorem}[{\sc Extended}]
  \label{T22}
For every $ \alpha \in (\underline{\alpha}, \overline{\alpha}]$ and
$p\in [p_0,q)$, the set $\mathcal{G}_{p,\alpha}$ is non-void and
weakly compact. For every $\lambda >0$ and $x\in {\sf V}$, there
exists a positive constant $C( p, \alpha;\lambda,x)$,  such that,
for  all $\mu \in \mathcal{G}_{p, \alpha}$,
\begin{equation}
  \label{A210}
 \int_{\mathit{\Omega}} \exp \left( \lambda |\omega(x)|^{p} \right)\mu({\rm d}\omega) \leq C(p, \alpha;\lambda,x).
\end{equation}
Furthermore, for every $\lambda >0$, there exists a positive
constant $C(p,\alpha;\lambda)$,  such that, for  all $\mu \in
\mathcal{G}_{p,\alpha}$,
\begin{equation}
  \label{A210A}
\int_{\mathit{\Omega}} \exp \left( \lambda \|\omega\|_{p,\alpha}^{p}
\right)\mu({\rm d}\omega) \leq C(p,\alpha;\lambda).
\end{equation}
\end{theorem}
Let us make some comments.
\begin{itemize}
\item
For our graphs, one cannot expect that the constants $C(p,\alpha;
\lambda,x)$ in (\ref{A210}) are bounded uniformly in  $x$. This
could be the case if the quantities  $\mathit{\Theta}(\alpha,
\theta)$ were bounded uniformly with respect to the  choice of the
root $o$.
\item
Both estimates (\ref{A210}) and (\ref{A210A}) hold also for $p = q$
but not for all $\lambda$, which should be small enough in this
case.
\item
The interval $[p_0,q)$ is non-void if $q> r + r/\theta$, i.e., if
the stabilizing effect of the potential $V$ is stronger than the
destabilizing  effects of the interaction and of the underlying
graph, caused by its degree property. If the graph is of bounded
degree $\bar{n} = \sup_{x} n(x)$, the condition (\ref{ass}) is
satisfied for any $\theta>0$ and $\alpha > \log \bar{n}$. In this
case, one can take $\theta$ arbitrarily big and get $q>r$ (or $q\geq
r$ for small $\lambda$), which is typical for such situations.
\item  According to (\ref{A12}) and (\ref{B25}),
 the stronger estimates we want to get,
the smaller class of tempered Gibbs random  fields we obtain.
\item In view  of the specific features of the graph geometry,
such as the degree unboundedness and the lack of transitivity, the
two basic statistical-mechanical tools -- Ruelle's superstability
method and Dobrushin's existence and uniqueness criteria -- are not
applicable to our model.
\end{itemize}

The proof of Theorem \ref{T22} will be done in Section \ref{PROsec}.
Theorem \ref{T2} is obtained therefrom as a particular case of
$\alpha = \overline{\alpha}$, $p=p_0$, and $r=2$.

\subsection{More on weak compactness}

Taking into account (\ref{B26}), we define
\begin{equation}
  \label{B27}
\widetilde{\mathit{\Omega}}^{\rm t} = \bigcap_{p\in [p_0,q)}
\bigcap_{\alpha \in (\underline{\alpha}, \overline{\alpha}] }  L^p
({\sf V}, w_\alpha).
\end{equation}
This set  can be endowed with the projective limit topology and
thereby turned into a Fr\'echet space. By standard arguments, its
Borel $\sigma$-field $\mathcal{B}(\widetilde{\mathit{\Omega}}^{\rm
t})$ has the property
\begin{equation} \label{B29}
\mathcal{B}(\widetilde{\mathit{\Omega}}^{\rm t}) = \{ A \cap
\widetilde{\mathit{\Omega}}^{\rm t} \ | \ A\in
\mathcal{B}(\mathit{\Omega})\},
\end{equation}
in view of which, we can define, cf. (\ref{t1}),
\begin{equation}
  \label{B28}
\widetilde{\mathcal{G}}^{\rm t} = \{ \mu \in \mathcal{G} \ | \ \mu
\left(\widetilde{\mathit{\Omega}}^{\rm t}\right)=1\}.
\end{equation}
The elements of the latter set have the smallest support we have
managed to establish. In view of (\ref{B29}), they can be redefined
as probability measures on $(\widetilde{\mathit{\Omega}}^{\rm t},
\mathcal{B}(\widetilde{\mathit{\Omega}}^{\rm t}))$. Let
$\widetilde{\mathcal{W}}^{\rm t}$ be the weak topology on the set of
all probability measures
$\mathcal{P}(\widetilde{\mathit{\Omega}}^{\rm t})$. Clearly,
$\widetilde{\mathcal{W}}^{\rm t}$ is stronger than the topology
mentioned in Theorem \ref{T22}.
\begin{theorem}
 \label{T3}
The set $\widetilde{\mathcal{G}}^{\rm t}$ is non-void and
$\widetilde{\mathcal{W}}^{\rm t}$-compact.
\end{theorem}
\begin{proof}
Let $\widetilde{\mathcal{G}}$ be the intersection of all
$\mathcal{G}_{p,\alpha}$, with $\alpha \in (\underline{\alpha},
\overline{\alpha}]$ and $\ p\in [p_0,q)$, see (\ref{B27}). By
compactness established in Theorem \ref{T22} the set
$\widetilde{\mathcal{G}}$ is non-void. Obviously, all its elements
belong to $\widetilde{\mathcal{G}}^{\rm t}$ and hence these two sets
coincide. Furthermore, the elements of $\widetilde{\mathcal{G}}$
obey the estimates (\ref{A210}), (\ref{A210A}) with all $\alpha \in
(\underline{\alpha}, \overline{\alpha}]$ and $\ p\in [p_0,q)$.

Let us now prove the stated $\widetilde{\mathcal{W}}^{\rm
t}$-compactness. To this end we consider the balls
\begin{equation} \label{ball}
B_{p,\alpha} (R) = \{ \omega \ | \ \|\omega\|_{p, \alpha} \leq R\},
\qquad R>0,
\end{equation}
and fix two monotone sequences $\alpha_k \downarrow
\underline{\alpha}$ and $p_k \uparrow q$, as $k\rightarrow +\infty$.
In view of (\ref{A210A}), for any $k\in \mathbb{N}$ and
$\epsilon>0$, one can pick $R_{k,\epsilon}>0$
 such that,
\[
\mu  \left[ B_{p_k,\alpha_k} (R_{k,\epsilon}) \right] \geq 1 -
\epsilon/2^k,
\]
uniformly for all $\mu \in \mathcal{G}_{p_k,\alpha_k}$, and hence
for all $\mu \in \widetilde{\mathcal{G}}^{\rm t}$. By the
compactness of the embedding (\ref{B26}), the set
\[
B =  \bigcap_{k\in \mathbb{N}}  B_{p_k,\alpha_k} (R_{k,\epsilon})
\]
is compact in $\widetilde{\mathit{\Omega}}^{\rm t}$, and is such
that $\mu(B) \geq 1 - \epsilon$ for all $\mu \in
\widetilde{\mathcal{G}}^{\rm t}$. Thereafter, the
$\widetilde{\mathcal{W}}^{\rm t}$-compactness of
$\widetilde{\mathcal{G}}^{\rm t}$ follows by the renowned Prokhorov
theorem.
\end{proof}

\subsection{Gibbs states of systems of anharmonic oscillators}
The Gibbs random fields constructed above can serve as equilibrium
thermodynamic states of systems of one-dimensional anharmonic
oscillators, indexed by the vertices of ${\sf G}$ and interacting
with each other along the edges by the potential $W$ (oscillating
networks). Obviously, Theorem \ref{T2} holds true if one replaces
the single-spin space $\mathbb{R}$ with $\mathbb{R}^\nu$, $\nu\in
\mathbb{N}$, which would correspond to multi-dimensional
oscillators. Furthermore, by means of the technique developed in
\cite{[Mono],[KP],[P]} this theorem can also be extended to the case
where the single-spin spaces are copies of $C_\beta$ -- the Banach
space of continuous functions (temperature loops) $\omega: [0,\beta]
\rightarrow \mathbb{R}^\nu$, $\beta >0$, such that $\omega (0) =
\omega (\beta)$. In this case, the Gibbs random fields correspond to
the so called {\it Euclidean} thermodynamic Gibbs states of a system
of interacting $\nu$-dimensional quantum anharmonic oscillators, for
which  $\beta^{-1}$ is temperature.

\section{Properties of the local Gibbs specification}

\label{PROsec} In this section, we prove that the estimate
(\ref{A210A}) holds also for all
$\pi_{\mathit{\Lambda}}(\cdot|\xi)$. This will imply all the
properties of the family
$\{\pi_{\mathit{\Lambda}}\}_{\mathit{\Lambda}\Subset {\sf V}}$ which
we need to prove Theorem \ref{T22}. We begin by deriving a basic
estimate, which allows us to control the $\xi$-dependence of moments
of $\pi_{\mathit{\Lambda}}$ with one-point  $\mathit{\Lambda} =
\{x\}$. Its extension to arbitrary $\mathit{\Lambda}$'s will be
obtained by means of the consistency property (\ref{A19a}).

\subsection{Moment estimates}
 From (\ref{A150}),  by an easy calculation we get
\begin{equation}
  \label{A24}
|W(u, v)| \leq \varkappa \left( |u|^p + |v|^p \right) + I_W/2 +
2(p-r) \left(\frac{J_W}{2p}\right)^{p/(p-r)}\left(\frac{r}{
\varkappa}\right)^{r /(p-r)},
\end{equation}
which holds for all $u,v \in \mathbb{R}$, and $\varkappa >0$, $p>r$.
We will use this estimate with $\varkappa = \beta / n(x) n(y)$,
$x,y\in {\sf V}$, $\beta >0$. For such $\beta$ and $p\in [p_0,q)$,
we set
\begin{eqnarray}
  \label{A25}
\mathit{\Gamma}_{xy}(\beta,p) & = &  \gamma(\beta,p) [n(x) n(y)]^{r
/(p-r)} ,\\[.2cm] \gamma(\beta,p)  &{=} &  I_W + 4(p-r)
\left(\frac{J_W}{2p}\right)^{p/(p-r)}
\left(\frac{r}{\beta}\right)^{r/(p-r)} , \nonumber
\end{eqnarray}
and
\begin{eqnarray}
  \label{A26}
 C({\beta,\lambda,p}) & = & c_V  +\log\left\{\int_{\mathbb{R}}
 \exp\left( ( \lambda +\beta) |u|^p - a_V |u|^q \right) {\rm d}u\right\}\\[.2cm] & - &
\log\left\{\int_{\mathbb{R}}
 \exp\left( - \beta |u|^p -V (u) \right) {\rm d}u\right\}, \nonumber
\end{eqnarray}
where $\lambda>0$ and $a_V$, $c_V$, and $q$ are the same as in
(\ref{A16}) and (\ref{p}). Note that the integral in the latter line
is positive. In the lemma below, $\pi_x$ and $Z_x$ stand for the
corresponding objects defined in (\ref{ass2}) with $\mathit{\Lambda}
= \{x\}$.
\begin{lemma}
  \label{1lm}
For every $\lambda>0$,  $p\in [p_0,q)$, $x\in {\sf V}$,  and $\xi\in
\mathit{\Omega}$, the following estimate holds
\begin{eqnarray}
  \label{A27}
& & \int_{\mathit{\Omega}} \exp\left( \lambda |\omega(x)|^p
\right)\pi_x
({\rm d}\omega  |\xi) \\[.2cm] & &
\qquad \quad \leq \exp\left(C(\beta,\lambda ,p)+ \sum_{y\sim x}
 \frac{2\beta|\xi(y)|^p}{n(x) n(y)}  + \sum_{y\sim x} \mathit{\Gamma}_{xy}(\beta ,p) \right). \nonumber
\end{eqnarray}
\end{lemma}
\begin{proof}
By (\ref{A24}), with $\varkappa = \beta / n(x) n(y)$, and
(\ref{A25}) the relative Hamiltonian (\ref{ass1}) with
$\mathit{\Lambda} = \{x\}$ can be estimated as follows
\begin{eqnarray*}
& &  - \sum_{y\sim x} \left[ \frac{\beta}{n(x) n(y)}
\left(|\omega(x)|^p+ |\xi(y)|^p \right) + \frac{1}{2}
\mathit{\Gamma}_{xy}(\beta ,p)\right] - V(\omega (x))\\[.3cm] & & \qquad
\qquad \qquad \qquad \qquad
 \leq  - H_x (\omega (x) |\xi) \leq \\[.3cm]  & &
\leq \sum_{y\sim x} \left[ \frac{\beta}{n(x) n(y)}
\left(|\omega(x)|^p+ |\xi(y)|^p \right) + \frac{1}{2}
\mathit{\Gamma}_{xy}(\beta ,p)\right]\\[.2cm] & &  - a_V |\omega(x)|^q + c_V.
\end{eqnarray*}
Then
\begin{eqnarray*}
Z_x (\xi) & \geq & \exp\left\{ - \sum_{y\sim x} \left[
\frac{\beta|\xi(y)|^p }{n(x) n(y)} + \frac{1}{2}
\mathit{\Gamma}_{xy}(\beta ,p) \right]\right\}\\[.2cm] & \times & \int_{\mathbb{R}} \exp\left[
- \beta |\omega(x)|^p - V(\omega (x))\right]{\rm d}\omega(x),
\end{eqnarray*}
and
\begin{eqnarray*}
& & \int_{\mathbb{R}}  \exp\left[ \lambda |\omega (x)|^p - H_x
(\omega (x) |\xi)\right]{\rm d}\omega(x) \\[.2cm] & & \qquad \leq \exp\left\{c_V+  \sum_{y\sim x} \left[
\frac{\beta|\xi(y)|^p }{n(x) n(y)} + \frac{1}{2}
\mathit{\Gamma}_{xy}(\beta ,p) \right]\right\}\\[.2cm] &  & \qquad \times \int_{\mathbb{R}} \exp\left[
( \lambda + \beta)|\omega(x)|^p - a_V|\omega (x)|^q\right]{\rm
d}\omega(x),
\end{eqnarray*}
which clearly yields (\ref{A27}).
\end{proof}
\noindent Now, for $\lambda>0$, $p\in [p_0,q)$, $\mathit{\Lambda}
\Subset {\sf V}$, and a fixed $x\in \mathit{\Lambda}$, we set
\begin{equation}
  \label{A28}
 M_x (\lambda,p,\mathit{\Lambda}; \xi) =  \log\left\{ \int_{\mathit{\Omega}} \exp\left( \lambda |\omega(x)|^p \right)
   \pi_{\mathit{\Lambda}}({\rm d}\omega|\xi)\right\},
\end{equation}
which is obviously finite. Our aim  is to find an upper bound for
this quantity. Integrating both sides of (\ref{A27}) with respect to
$\pi_{\mathit{\Lambda}} (\cdot|\xi)$ and taking into account
(\ref{A19a}) we obtain
\begin{eqnarray}
  \label{A29}
\qquad \exp\left[ M_x (\lambda,p,\mathit{\Lambda}; \xi) \right]&
\leq & \exp\left( C(\beta,\lambda,p) + \sum_{y\sim x }
\mathit{\Gamma}_{xy}(\beta ,p)\right. \\[.2cm] & + & \left. \sum_{y\sim x, \ y\in
\mathit{\Lambda}^c}\frac{2 \beta
|\xi(y)|^p}{n(x)n(y) }\right)\nonumber \\[.2cm] & \times &
\int_{\mathit{\Omega}}\exp\left(\sum_{y\sim x, \ y\in
\mathit{\Lambda}}\frac{2 \beta |\omega(y)|^p}{n(x)n(y) }
\right)\pi_{\mathit{\Lambda}} ({\rm d}\omega |\xi).\nonumber
\end{eqnarray}
In the sequel, the parameter $\overline{\alpha}$ will be fixed. Then
for a given $\lambda$, the parameter $\beta$ will always be  chosen
in such a way that
\begin{equation}
  \label{A47}
2 \beta e^{\overline{\alpha}} < \lambda,
\end{equation}
which, in particular, yields
\begin{equation}
  \label{A30} \sum_{y\sim x} \frac{2\beta}{ \lambda n(x) n(y) } \leq 1.
\end{equation}
To estimate the integral in the latter line in (\ref{A29}) we use
the multiple H\"older inequality
\begin{equation} \label{hi}
\int \left( \prod_{i=1}^n \varphi_i^{\alpha_i} \right) {\rm d}\mu
\leq \prod_{i=1}^n  \left(\int \varphi_i {\rm d}\mu
\right)^{\alpha_i},
\end{equation}
in which $\mu$ is a probability measure, $\varphi_i \geq 0$
(respectively, $\alpha_i \geq 0$), $i=1, \dots , n$, are integrable
functions (respectively, numbers such that $\sum_{i=1}^n \alpha_i
\leq 1$). Applying this inequality in (\ref{A29}) and taking into
account (\ref{A30}) we arrive at
\begin{eqnarray}
\label{A31} \quad \qquad M_x (\lambda,p,\mathit{\Lambda}; \xi) &
\leq & C(\beta,\lambda,p) + \sum_{y\sim x}\mathit{\Gamma}_{xy}(\beta
,p) +\sum_{y\sim x, \ y\in
\mathit{\Lambda}^c}\frac{2 \beta |\xi(y)|^p}{n(x)n(y) } \\[.2cm]& + & \sum_{y\sim
x, \ y\in \mathit{\Lambda}}\frac{2 \beta }{\lambda n(x)n(y) }M_y
(\lambda,p,\mathit{\Lambda}; \xi).\nonumber
\end{eqnarray}
As the quantity we want to estimate appears in both sides of the
latter estimate, we make the following. For $\alpha \in
(\underline{\alpha},\overline{\alpha}]$, we set, cf. (\ref{A11}) and
(\ref{A12}),
\begin{equation}
  \label{A32}
\|M(\lambda,p,\mathit{\Lambda};\xi)\|_\alpha = \sum_{x\in
\mathit{\Lambda}}M_x (\lambda,p,\mathit{\Lambda}; \xi) \exp[- \alpha
\rho(o,x)],
\end{equation}
and obtain an upper bound for
$\|M(\lambda,p,\mathit{\Lambda};\xi)\|_\alpha$. To this end we
multiply both sides of (\ref{A31}) by $\exp[-\alpha \rho(o,x)]$ and
sum over $x\in \mathit{\Lambda}$. This leads us to
\begin{equation}
\label{A41} \|M(\lambda,p,\mathit{\Lambda};\xi)\|_{\alpha} \leq
\mathit{\Upsilon}^\alpha_1 + \mathit{\Upsilon}^\alpha_2 +
\mathit{\Upsilon}^\alpha_3 (\mathit{\Lambda}) +
\mathit{\Upsilon}^\alpha_4 (\mathit{\Lambda}).
\end{equation}
Here
\begin{equation}
  \label{A41a}
\mathit{\Upsilon}^\alpha_1 = C(\beta,\lambda,p) \sum_{x}\exp[-\alpha
\rho (o,x)],
\end{equation}
and
\begin{eqnarray}
  \label{A42}
\mathit{\Upsilon}^\alpha_2 = \gamma(\beta,p) \mathit{\Theta} (\alpha
; r/(p-r))\leq \gamma(\beta,p) \mathit{\Theta} (\alpha ; \theta).
\end{eqnarray}
The latter estimate holds since $p\geq p_0 = r+ r/\theta$. The term
corresponding to the third summand in (\ref{A31}) is estimated as
follows
\begin{eqnarray}
  \label{A45}
& & \quad  \sum_{x \in \mathit{\Lambda}} \exp[-\alpha \rho
(o,x)]\sum_{y \sim x, \ y \in \mathit{\Lambda}^c}
\frac{2\beta}{n (x) n (y)}|\xi(y)|^{p} \\[.2cm] & & \qquad
 \qquad \qquad  \leq \mathit{\Upsilon}^\alpha_3 (\mathit{\Lambda})
\ \stackrel{\rm def}{=} \   2 \beta e^\alpha \sum_{x \in
\mathit{\Lambda}^c}\exp[-\alpha \rho (o,x)]|\xi(x)|^{p} ,\nonumber
\end{eqnarray}
which is finite whenever $\xi\in L^p({\sf V}, w_\alpha)$, and tends
to zero as $\mathit{\Lambda} \rightarrow {\sf V}$.
 In a similar way, we get
\begin{eqnarray}
  \label{A46} & & \quad  \sum_{x \in \mathit{\Lambda}} \exp[-\alpha \rho
(o,x)] \sum_{y \sim x, \ y \in \mathit{\Lambda}}
\frac{2\beta}{\lambda n(x) n(y)}M_{y} (\lambda,p,\mathit{
 \Lambda}; \xi) \\[.2cm] & & \qquad
\qquad \qquad  \leq \mathit{\Upsilon}^\alpha_4 (\mathit{\Lambda}) \
\stackrel{\rm def}{=} \
 \frac{2\beta e^\alpha}{\lambda} \|M (\lambda,p,\mathit{
 \Lambda}; \xi)\|_\alpha.\nonumber
\end{eqnarray}
Recall that $\beta$ and $\lambda$ are supposed to obey (\ref{A47}).
Then from the estimates obtained above we get the following
\begin{equation}
  \label{A48}
\|M(\mathit{\lambda,p,\Lambda};\xi)]\|_\alpha \leq
\frac{\mathit{\Upsilon}^\alpha_1 + \mathit{\Upsilon}^\alpha_2 +
\mathit{\Upsilon}^\alpha_3 (\mathit{\Lambda})}{1 - 2 \beta e^\alpha
/\lambda},
\end{equation}
which yields
\begin{equation}
  \label{AA}
M_x(\mathit{\lambda,p,\Lambda};\xi)\leq C(\lambda, p, x, \xi),
\end{equation}
for some $C(\lambda, p, x, \xi)>0$, which is independent of
$\mathit{\Lambda}$, but obviously depends on the choice of the root
$o$.

\subsection{Weak compactness of the local Gibbs specification}

The result just obtained allows us to prove the next statement,
crucial for establishing the relative weak compactness of the family
$\{\pi_{\mathit{\Lambda}}(\cdot |\xi)\}_{\mathit{\Lambda}\Subset
{\sf V}}$ and the corresponding integrability estimates.
\begin{lemma}
  \label{A2lm}
Let $p\in [p_0,q)$ and $\alpha \in (\underline{\alpha},
\overline{\alpha}]$ be fixed. Then for every $\lambda >0$ and  $\xi
\in L^{p}({\sf V}, w_\alpha)$, one finds a positive constant
 $C(p,\alpha; \lambda, \xi)$, such that for all ${\mathit{\Lambda}}\Subset {\sf V}$,
 \begin{equation}
   \label{A51}
\int_{\mathit{\Omega}} \exp\left( \lambda
\|\omega\|^p_{p,\alpha}\right) \pi_{\mathit{\Lambda}} ({\rm d}\omega
|\xi) \leq C(p,\alpha; \lambda,\xi).
 \end{equation}
Furthermore, for the same $\lambda$, one finds a positive constant
$C( p,\alpha;\lambda)$, such that for all $\xi \in L^{p}({\sf V},
w_\alpha)$,
\begin{equation}
   \label{A520}
\limsup_{\mathit{\Lambda}\rightarrow {\sf V}}\int_{\mathit{\Omega}}
\exp\left( \lambda \|\omega\|^p_{p,\alpha}\right)
\pi_{\mathit{\Lambda}} ({\rm d}\omega |\xi) \leq C(p,\alpha;
\lambda).
\end{equation}
\end{lemma}
\begin{proof}
By (\ref{ass2}) and (\ref{A12}), for any $\delta >0$, we have
\begin{eqnarray}
 \label{ass3}
& & \int_{\mathit{\Omega}} \exp\left( \lambda
\|\omega\|^p_{p,\alpha}\right) \pi_{\mathit{\Lambda}} ({\rm d}\omega
|\xi) = \exp \left(\lambda \sum_{x\in \mathit{\Lambda}^c}
|\xi(x)|^p w_\alpha (x) \right)\\[.2cm] & & \quad \qquad \times \int_{\mathit{\Omega}}
\prod_{x\in \mathit{\Lambda}}\left[\exp\left( \delta |\omega
(x)|^p\right) \right]^{\lambda w_\alpha (x)/\delta}
\pi_{\mathit{\Lambda}} ({\rm d}\omega |\xi).  \nonumber
\end{eqnarray}
Now we pick $\delta$, such that
\[
\frac{\lambda}{\delta}\sum_{x\in \mathit{\Lambda}} w_\alpha (x) \leq
1,
\]
and apply in (\ref{ass3}) the H\"older inequality (\ref{hi}). This
yields, see (\ref{A28}) and (\ref{A32}),
\begin{eqnarray} \label{A52}
 \int_{\mathit{\Omega}} \exp\left( \lambda
\|\omega\|^p_{p,\alpha}\right) \pi_{\mathit{\Lambda}} ({\rm d}\omega
|\xi)& \leq & \exp \left(\lambda \sum_{x\in \mathit{\Lambda}^c}
|\xi(x)|^p w_\alpha (x) \right)\\[.2cm] &\times& \exp\left[ (\lambda /\delta)
\|M(\delta,p,\mathit{\Lambda};\xi) \|_\alpha\right]. \nonumber
\end{eqnarray}
By (\ref{A48}) the set $\{{\rm
RHS(\ref{A52})}(\mathit{\Lambda})|\mathit{\Lambda}\Subset {\sf V}\}$
is bounded for every fixed $\xi\in L^{p}({\sf V}, w_\alpha)$. We
denote its upper bound by $C(p,\alpha; \lambda , \xi)$ and obtain
(\ref{A51}). The estimate (\ref{A520}) follows from (\ref{A52}) by
(\ref{A45} ), (\ref{A48}), and the fact that $\xi\in L^{p}({\sf V},
w_\alpha)$.
\end{proof}
\begin{corollary}
  \label{assco}
For every $\xi\in L^{p}({\sf V}, w_\alpha)$, the family
$\{\pi_{\mathit{\Lambda}}(\cdot|\xi)\}_{\mathit{\Lambda}\Subset {\sf
V}}\subset \mathcal{P}(\mathit{\Omega})$ is relatively weakly
compact.
\end{corollary}
\begin{proof}
For obvious reasons, the balls $\{\omega \ | \ \|\omega\|_{p,\alpha}
\leq R\}$, $R>0$, are compact in $\mathit{\Omega}$ for any fixed
$\alpha \in (\underline{\alpha}, \overline{\alpha}]$ and $p\in
[p_0,q)$. Thus, the proof follows from (\ref{A51}) by  Prokhorov's
theorem.
\end{proof}

We recall that $C_{\rm b} (\mathit{\Omega})$ stands for the set of
bounded continuous functions $f:\mathit{\Omega} \rightarrow
\mathbb{R}$.
 To prove that the accumulation points of the family
$\{\pi_{\mathit{\Lambda}} (\cdot|\xi)\}_{\mathit{\Lambda}\Subset
{\sf V}}\subset \mathcal{P}(\mathit{\Omega})$ are Gibbs measures, we
use the fact that this family possesses the following (Feller)
property. For a fixed $\mathit{\Lambda}\Subset {\sf V}$, we consider
\begin{equation}
  \label{map}
C_{\rm b} (\mathit{\Omega}) \ni f \mapsto \pi_{\mathit{\Lambda}}
(f|\cdot) \ \stackrel{\rm def}{=} \
\int_{\mathit{\Omega}}f(\omega)\pi_{\mathit{\Lambda}} ({\rm
d}\omega|\cdot).
\end{equation}
\begin{lemma}
  \label{asslm}
For every $\mathit{\Lambda}\Subset {\sf V}$, (\ref{map}) maps
$C_{\rm b} (\mathit{\Omega})$ into itself.
\end{lemma}
The proof of this lemma is quite standard. The boundedness of
$\pi_{\mathit{\Lambda}}(f|\cdot)$ is immediate. Its continuity
 follows from the continuity of $W$, see Assumption \ref{ASS}, and the estimates (\ref{A48}) and
(\ref{hi}) by Lebesgue's dominated convergence theorem. For more
details, we refer the reader to the proof of the corresponding lemma
in \cite{[KP]}.

\begin{corollary}
  \label{asscon}
For every $p\in [p_0, q)$ and $\alpha \in (\underline{\alpha},
\overline{\alpha}]$, the set $\mathcal{G}_{p,\alpha}$ is non-void.
\end{corollary}
\begin{proof}
For every $\mathit{\Lambda}\Subset {\sf V}$ and $\xi \in L^{p}({\sf
V}, w_\alpha)$, by (\ref{ass2}) each $\pi_{\mathit{\Lambda}} (\cdot
|\xi)$ is supported by the set \[\{\omega =
\omega_{\mathit{\Lambda}} \times \xi_{\mathit{\Lambda}^c}  \ | \
\omega_{\mathit{\Lambda}} \in
\mathit{\Omega}_{\mathit{\Lambda}}\},\] which yields
\begin{equation}
  \label{Z}
\pi_{\mathit{\Lambda}} \left[L^{p}({\sf V}, w_\alpha)|\xi\right] =1.
\end{equation}
Let us fix some $\xi\in L^{p}({\sf V}, w_\alpha)$.  By Corollary
\ref{assco} there exists an increasing sequence
$\{\mathit{\Lambda}_n\}_{n\in \mathbb{N}}$, which exhausts ${\sf
V}$, such that $\{\pi_{\mathit{\Lambda}_n}(\cdot |\xi)\}_{n\in
\mathbb{N}}$ weakly converges to a certain $\mu \in
\mathcal{P}(\mathit{\Omega})$. Let us show that this $\mu$ also
solves the DLR equation. For any $\mathit{\Lambda}$, one finds
$n'\in \mathbb{N}$, such that $\mathit{\Lambda}\subset
\mathit{\Lambda}_{n}$ for all $n\geq n'$. For such $n$ and $f\in
C_{\rm b}(\mathit{\Omega})$, by (\ref{A19a}) we have
\begin{equation} \label{fel}
\int_{\mathit{\Omega}}
\pi_\mathit{\Lambda}(f|\omega)\pi_{\mathit{\Lambda}_n} ({\rm
d}\omega|\xi) = \pi_{\mathit{\Lambda}_n} (f|\xi).
\end{equation}
Then we pass here to the limit $n\rightarrow +\infty$ and obtain
that $\mu\in \mathcal{G}$, see (\ref{dlr1}) and Lemma \ref{asslm}.
To prove that $\mu$ is supported by $L^p({\sf V}, w_\alpha)$ we show
that this measure obeys the estimate (\ref{A210A}). For $\lambda>0$,
we set
\begin{equation}
  \label{A63}
F_N (\omega) = \exp\left[\lambda \min\left\{\|\omega\|^p_{p,\alpha};
N \right\} \right], \quad \ N\in \mathbb{N},
\end{equation}
which is a lower semi-continuous function on $\mathit{\Omega}$. Then
by (\ref{A520}) and the weak convergence
$\pi_{\mathit{\Lambda}_n}(\cdot|\xi) \rightarrow \mu$, we have
\begin{eqnarray} \label{FE}
\int_{\mathit{\Omega}} F_N (\omega) \mu({\rm d}\omega) & \leq &
\lim_{n\rightarrow +\infty} \int_{\mathit{\Omega}} F_N (\omega)
\pi_{\mathit{\Lambda}_n}({\rm d}\omega|\xi)\\[.2cm] & \leq & C(p,\alpha;
\lambda),\nonumber
\end{eqnarray}
where the latter constant is the same as in (\ref{A520}).
Thereafter, the proof of (\ref{A210A}), with the same constant,
follows by  B. Levi's monotone convergence theorem. Hence, $\mu \in
\mathcal{G}_{p,\alpha}$.
\end{proof}

{\sc Proof of Theorem \ref{T22}.} Just above we have proven that the
accumulation points of the family $\{\pi_{\mathit{\Lambda
}}(\cdot|\xi)\}$, $\xi \in L^p({\sf V}, w_\alpha)$, obey
(\ref{A520}). Let us extend this to all $\mu \in
\mathcal{G}_{p,\alpha}$. For such $\mu$, by (\ref{dlr}), Fatou's
lemma, and the estimate (\ref{A520}) we get
\begin{eqnarray*}
 & & \int_{\mathit{\Omega}}F_N (\omega)\mu({\rm d}\omega) =
 \limsup_{\mathit{\Lambda}\rightarrow {\sf V}}
\int_{\mathit{\Omega}} \left[\int_{\mathit{\Omega}}F_N (\omega)
\pi_{\mathit{\Lambda}} ({\rm d}\omega|\xi) \right]\mu({\rm d}\xi)
\\[.2cm] & & \ \qquad \leq
\int_{\mathit{\Omega}}\left[ \limsup_{\mathit{\Lambda}\rightarrow
{\sf V}} \int_{\mathit{\Omega}}F_N (\omega) \pi_{\mathit{\Lambda}}
({\rm d}\omega|\xi) \right]\mu({\rm d}\xi)
\\[.2cm]  & & \ \qquad \leq \int_{\mathit{\Omega}}
  \left[ \limsup_{\mathit{\Lambda}\rightarrow {\sf V}} \int_{\mathit{\Omega}}
  \exp\left(\lambda \|\omega\|^p_{p,\alpha}\right)\pi_{\mathit{\Lambda}}
  ({\rm d}\omega|\xi) \right]\mu({\rm d}\xi)\\[.2cm] & & \ \qquad \leq C(p,\alpha;\lambda).
\end{eqnarray*}
Then we again apply B. Levi's theorem and obtain (\ref{A210A}). The
proof of (\ref{A210}) follows by (\ref{AA}) along the same line of
arguments.

In view of (\ref{A210A}), by Prokhorov's theorem  the set
$\mathcal{G}_{p,\alpha}$ is relatively weakly compact. Clearly, all
its accumulation points solve the DLR equation (\ref{dlr1}); hence,
$\mathcal{G}_{p,\alpha}$ is weakly compact.
 $\square$

\section{Repulsive graphs}
\label{RGsec}

 In the remaining part of the paper, we present a
family of unbounded degree graphs, which obey  the estimate
(\ref{ass}). A crucial property of such graphs is that vertices of
large degree are located at large distances from each other.

\subsection{The family of graphs and the main statement}
For $n_* \in \mathbb{N}$, we set
\begin{equation}
  \label{NU}
{\sf V}_* = \{ x\in {\sf V} \ | \ n(x) \leq n_*\}, \qquad {\sf
V}_*^c = {\sf V} \setminus {\sf V}_*.
\end{equation}
\begin{definition}
For an integer $n_*>2$ and a strictly increasing function $\phi:
(n_*, + \infty) \rightarrow (0, + \infty)$, the family
$\mathbb{G}(n_*, \phi)$ consists of those graphs ${\sf G}= ({\sf V},
{\sf E})$, for which
 the path distance obeys the condition
\begin{equation} \label{I}
\forall x,y\in {\sf V}_*^c: \qquad \rho(x,y) \geq \phi [n(x,y)],
\end{equation}
where
\begin{equation}
  \label{Qq0}
  n(x,y) = \max\{n(x); n(y)\}.
\end{equation}
No restrictions are imposed on $\rho(x,y)$ if either $x$ or $y$
belongs to ${\sf V}_*$.
\end{definition}
Let us make some comments. For a given $x\in {\sf V}_*^c$, for
\begin{equation*}
   {\sf K} (x)\
\stackrel{\rm def}{ =} \  \{ y \in {\sf V} \ | \ \rho (y,x) < \phi
[n(x)]\},
\end{equation*}
 by
(\ref{I}) one has that ${\sf K} (x){\cap} {\sf V}_*^c = \{x\}$,
i.e., such $x$  `repels' all  vertices $y\in{\sf V}_*^c$ from the
ball ${\sf K} (x)$. For the sake of convenience, we shall assume
that ${\sf K}(x)$ contains the neighborhood of $x$, which is
equivalent to assuming that
\begin{equation}
\label{KK} \phi (n_*+1) > 1.
\end{equation}
The graphs introduced and studied in \cite{[BD]} were defined by a
condition, which can be written in the form, cf. eqs. (3.8) and
(3.9) in \cite{[BD]},
\begin{equation} \label{NU3}
\rho(x,y) \geq \phi [m(x,y)], \qquad m(x,y) \ \stackrel{\rm def}{=}
\ \min\{ n(x); n(y)\}.
\end{equation}
In this case, a vertex $x$ `repels' from the ball $\{ y | \rho(y,x)
< \phi[n(x)]\}$ only those $y$'s, for which $n(y) \geq n(x)$. We
employ (\ref{I}) rather than (\ref{NU3}) in view of its application
in Lemma \ref{Jlm} below, see Remark \ref{NUrk} for further
comments. The concrete choice of the function $\phi$ in Theorem
\ref{Itm} is discussed in Remark \ref{NU1rk} below.
\begin{theorem}
 \label{Itm}
Let ${\sf G}$ be in  ${\sf G} \in \mathbb{G}(n_*,\phi)$ with $\phi$
having the form
\begin{equation} \label{march}
\phi (b) = \upsilon \log b \left[ \log\log b\right]^{1+\varepsilon},
\quad \ \ \upsilon, \varepsilon >0, \ \ \  b \geq n_* + 1,
\end{equation}
where $\upsilon$ and $\varepsilon$ are such that (\ref{KK}) holds.
Then for any $\theta>0$, there exists $\underline{\alpha} \geq0$,
which may depend on $\theta$, $n_*$, $\upsilon$, and $\varepsilon$,
such that $\mathit{\Theta} (\alpha , \theta) < \infty$ whenever
$\alpha > \underline{\alpha}$.
\end{theorem}
The proof of Theorem \ref{Itm} is given at the very end of this
subsection. It is preceded by and based on Lemmas \ref{Klm} and
\ref{Jlm}, which in turn are proven  in the remaining part of the
paper.
 For
$N\in \mathbb{N}$ and $x\in {\sf V}$, we set
\begin{eqnarray}
 \label{J1}
{\sf S} (N,x) & = & \{ y\in {\sf V} \ | \ \rho (x,y)= N\},\\[.2cm]
{\sf B} (N,x) & = & \{ y\in {\sf V} \ | \ \rho (x,y)\leq N\},
\nonumber
\end{eqnarray}
and
\begin{equation}
  \label{J0}
T_x (\alpha , \theta) = \sum_{y} [n(y)]^{1+\theta}
\exp\left[-\alpha\rho(x,y)\right], \quad \ \ \alpha,  \theta >0.
\end{equation}
\begin{lemma} \label{Klm}
Let ${\sf G}$ be in $\mathbb{G}(n_*, \phi)$ with $\phi$ obeying
(\ref{KK}). Then for every positive $\theta$ and $\alpha$,
\begin{equation}
 \label{J}
\mathit{\Theta} (\alpha , \theta) \leq n_*^\theta (e^\alpha +1)T_o
(\alpha , \theta).
\end{equation}
\end{lemma}
\begin{lemma}
 \label{Jlm}
Let ${\sf G}$ be as in Theorem \ref{Itm}. Then for every $\theta
>0$, there exists $a>0$, which may depend also on the parameters of
the function (\ref{march}), such that, for any $x\in {\sf V}$, there
exist $N_x\in \mathbb{N}$, for which
\begin{equation}
 \label{J2}
G_\theta (N,x)\ \stackrel{\rm def}{=} \ \sum_{y\in {\sf S} (N,x)}
[n(y)]^{1+\theta} \leq \exp \left( a N\right),
\end{equation}
whenever $N\geq N_x$.
\end{lemma}
\begin{remark}
  \label{NUrk}
A condition like (\ref{NU3}) could guarantee that the estimate
(\ref{J2}) holds only for $N=N_k$, $k\in \mathbb{N}$,  for some
increasing sequence $\{N_k\}_{k\in \mathbb{N}}\subset \mathbb{N}$.
This would not be  enough for proving Theorem \ref{Itm}.
\end{remark}

{\sc Proof of Theorem \ref{Itm}.}
 By (\ref{J0}) and (\ref{J2})
we have
\begin{eqnarray*} T_x (\alpha , \theta) & \leq &
\sum_{N=0}^{N_x}
 \exp (-\alpha N) \left(\sum_{y\in {\sf S} (N,x)} [n(y)]^{1+\theta}\right)\\[.2cm]
 & + &  \sum_{N=N_x+1}^{\infty} \exp [-(\alpha -a) N].
\end{eqnarray*}
Thus, the proof of the theorem follows by (\ref{J}) with
$\underline{\alpha} = a$. $\square$

\subsection{A property of the balls in repulsive graphs}
The proof of Lemma \ref{Jlm} is based on a property of the balls
${\sf B}(N,x)$ in the graphs ${\sf G}\in \mathbb{G}(n_*,\phi)$, due
to which one can control the growth of the maximum degree of $y\in
{\sf B}(N,x)$. Here we do not suppose that $\phi$ has the concrete
form of (\ref{march}).
\begin{lemma} \label{MAlm}
Let ${\sf G} = ({\sf V}, {\sf E})$ be in $\mathbb{G}(n_*,\phi)$ with
an arbitrary increasing function $\phi:(n_*, +\infty) \rightarrow
(1,+\infty)$. Then, for every $x\in {\sf V}$, there exists $N_x\in
\mathbb{N}$, such that
\begin{equation}
  \label{AP}
 \max_{y\in {\sf B}(N,x)} n(y) \leq \phi^{-1}(2N),
\end{equation}
whenever $N\geq N_x$.
\end{lemma}
\begin{proof}
Given $x$, let $\tilde{x}$ be the vertex in ${\sf V}_*^c$ which is
closest to $x$, see (\ref{NU}). If there are more than one such
vertices at the same distance, we take the one with the highest
degree. For this $\tilde{x}$, we have the following possibilities:
(i) $\rho(x, \tilde{x}) \geq \phi[n(\tilde{x})]/2$; (ii) $\rho(x,
\tilde{x}) < \phi[n(\tilde{x})]/2$. The latter one includes also the
case $\tilde{x} = x$, i.e., where $x$ itself is in ${\sf V}_*^c$. In
case (i), we set $N_x= 1$, which means that (\ref{AP}) holds for all
$N\in \mathbb{N}$. Indeed, if $N< \rho (x,\tilde{x})$, then the ball
${\sf B}(N,x)$ contains only vertices $y\in{\sf V}_*$, for which
$n(y) \leq n_* \leq \phi^{-1} (2N)$ for any $N\in \mathbb{N}$. If
$N\geq \rho (x,\tilde{x})$ and $\max_{y\in {\sf B}(N,x)} n(y) =
n(\tilde{x})$, one has $N \geq \rho(x, \tilde{x}) \geq
\phi[n(\tilde{x})]/2$, which yields (\ref{AP}) also for this case.
Finally, let $\max_{y\in {\sf B}(N,x)} n(y) = n(z)$ for some $z\neq
\tilde{x}$, which means that $n(z) > n(\tilde{x})$. In this case, by
(\ref{I}) we have $\rho (\tilde{x}, z) \geq \phi[(n(z)]$, and
\begin{equation} \label{esti}
2N \geq \rho (x, z) + \rho (x, \tilde{x}) \geq \rho (\tilde{x}, z)
\geq \phi[(n(z)],
\end{equation}
which yields (\ref{AP}) for this case as well.

If  (ii) holds, we  let  $x_1$ be the closest vertex to $x$, such
that $n(x_1)
> n(\tilde{x})$. Again, we take that of the highest degree if
there are more than one such vertices. By (\ref{I}) we have
$\rho(\tilde{x}, x_1) \geq \phi[n(x_1)]$.  If for $N\geq N_x \
\stackrel{\rm def}{=} \  \rho (x, x_1) $, one has $\max_{y\in {\sf
B}(N,x)} n(y) = n(x_1)$, then
\begin{eqnarray*} 
N &\geq & \rho (x, x_1) \geq \phi [n(x_1)] - \rho (x, \tilde{x})
\geq \phi [n(x_1)] - \phi [n(\tilde{x})]/2 \\[.2cm] & \geq & \phi
[n(x_1)]/2, \nonumber
\end{eqnarray*}
which yields (\ref{AP}). Finally, let $\max_{y\in {\sf B}(N,x)} n(y)
= n(z)$ for some $z\neq {x}_1$, which means that $n(z) > n({x}_1)$.
In this case, $\rho ({x}_1, z) \geq \phi[(n(z)]$, and we obtain
(\ref{AP}) by applying (\ref{esti}) with $\tilde{x}$ replaced by
$x_1$.
 \end{proof}

\subsection{Proof of Lemmas \ref{Klm} and \ref{Jlm}}
First we prove an auxiliary statement. Recall that by $\vartheta
(x,y)$ we denote a path with endpoints $x$ and $y$.   A path is
called {\it simple} if none of its inner vertices are repeated. For
$m\leq n$, let $\vartheta' = \{x_0, \dots , x_m\}$ and $\vartheta=
\{y_0, \dots , y_n\}$ be such that $x_0 = y_k$, $x_1 = y_{k+1},
\dots, x_m = y_{k+m}$ for some $k=0, \dots , n-m$. Then we say that
$\vartheta'$ is a {\it subpath} of $\vartheta$, and write
$\vartheta' \subset \vartheta$. For a path $\vartheta$, by ${\sf
V}_\vartheta$ we denote the set of all its vertices.

Let $\mathit{\Sigma}_N(x)$ denote the family of all simple paths of
length $N$ originated at $x$. Then, for every $y\in {\sf S}(N,x)$,
there exists $\vartheta \in \mathit{\Sigma}_N(x)$, such that
$\vartheta = \vartheta (x,y)$. We use this fact for estimating the
cardinality of ${\sf S}(N,x)$.
\begin{proposition}[cf. Assertion 6 of \cite{[BD]}]
  \label{J1pn}
In any graph ${\sf G}$, for any $x\in {\sf V}$ and $N\in
\mathbb{N}$, one has
\begin{equation}
  \label{J4}
|{\sf S}(N, x)| \leq |\mathit{\Sigma}_N(x)| \leq \max_{\vartheta \in
\mathit{\Sigma}_N(x)}  \prod_{y\in {\sf V}_\vartheta \setminus
\{x_N\}} n(y).
\end{equation}
\end{proposition}
\begin{proof}
The proof will be done by induction in $N$. For $N=1$, the estimate
(\ref{J4}) is obvious.  For any $N\geq 2$, we have
\begin{equation}
  \label{Ma3}
|\mathit{\Sigma}_N (x)| \leq n(x) \max_{y\sim x}
|\mathit{\Sigma}^x_{N-1} (y)|,
\end{equation}
where $\mathit{\Sigma}^x_{N-1} (y)$ is the corresponding family of
paths in the graph which one obtains from ${\sf G}$ be deleting the
edge $\langle x, y\rangle$.
 Every  $\vartheta \in \mathit{\Sigma}_N (x)$ can be written
in the form $\vartheta = \{ x\tilde{\vartheta}\}$ with
$\tilde{\vartheta} \in \mathit{\Sigma}^x_{N-1} (y)$ for some $y\sim
x$. Then by the inductive assumption we have
\begin{eqnarray*}
|\mathit{\Sigma}_N (x)|  & \leq & n(x) \max_{y \sim x}
\max_{\tilde{\vartheta} \in \mathit{\Sigma}^x_{N-1} (y)}
 \prod_{z\in
{\sf V}_{\tilde{\vartheta}} \setminus \{x_N\}} n(z) \\[.2cm] & \leq &
\max_{\vartheta \in \mathit{\Sigma}_{N} (x)}  \prod_{z\in {\sf
V}_{{\vartheta}} \setminus \{x_N\}} n(z),
\end{eqnarray*}
that completes the proof.
\end{proof}

{\sc Proof on Lemma \ref{Jlm}.} We are going to prove that the
estimate (\ref{J2}) holds with $N_x$ being as in Lemma \ref{MAlm}
and $a$ given by
\begin{eqnarray} \label{Mak6}
a & = & (1+\theta) \sigma + \log n_* + \frac{2e}{\upsilon}
\sum_{k=k_*}^{\infty} \frac{1}{k^{1+\varepsilon}},\\[.2cm]
\sigma & \stackrel{\rm def}{=} &
 \max\left\{2 / \upsilon ; e\right\}. \nonumber
\end{eqnarray}
For any $N\geq N_x$, by (\ref{AP}) and (\ref{J4}) we obtain
\begin{equation}
  \label{Mak2}
G_\theta (N, x) \leq \exp\left\{ (1+\theta) \log \phi^{-1} (2N) +
\max_{\vartheta \in \mathit{\Sigma}_N (x)} \sum_{z\in {\sf
V}_\vartheta \setminus \{x_N \}} \log n(z)\right\}.
\end{equation}
By (\ref{march}) we have
\begin{equation}
  \label{Mak3}
\phi^{-1} (2N) \leq \exp(\sigma N),
\end{equation}
where $\sigma$ is as in (\ref{Mak6}). If $x$ is as in the case (i)
considered in the proof of Lemma \ref{MAlm}, and $N <
\phi(n_*+1)/2$, then ${\sf V}_\vartheta \subset {\sf V}_*$ for any
$\vartheta \in \mathit{\Sigma}_N (x)$. In this case, the second
summand in $\{ \cdot \}$ in (\ref{Mak2}) does not exceed $N\log
n_*$, which certainly yields (\ref{J2}).

To handle the case of $N\geq \phi(n_*+1)/2$ we use the sequence
 $\{c_k\}_{k\in \mathbb{N}}$, where $c_k = \exp ( e^k)$, $k\in \mathbb{N}$.
Let $k_*$ be the least $k\in \mathbb{N}$ such that $c_{k_* +1} \geq
n_* + 1$. Then we set $b_{k_*} = n_* + 1$ and $b_k = c_k$ for $k
> k_*$. Let  $k_N$ be the largest
$k$, such that $b_{k} \leq \phi^{-1} (2N)$. For $k = k_*, \dots,
k_N$ and a given $\vartheta \in \mathit{\Sigma}_N (x)$, let
$m_k^\vartheta$ be the number of vertices $y\in {\sf V}_\vartheta$,
such that $n(y) \in [b_k, b_{k+1}]$. Given $\tau\in (0, N)$, for any
$\vartheta \in \mathit{\Sigma}_N (x)$, the number of vertices in
${\sf V}_\vartheta$ which  are away from each other at distance at
least  $\tau$  is  $1 + N / \tau$, at most. Therefore,
\begin{equation*}
 m_k^\vartheta \leq   m_k \ \stackrel{\rm def}{=} \ 1 + N/ \phi(b_k)\leq 2N / \phi (b_k).
\end{equation*}
Taking this into account by (\ref{march}) we get
\begin{eqnarray*}
  \max_{\vartheta \in \mathit{\Sigma}_N (x)} \sum_{z\in {\sf
V}_\vartheta \setminus \{x_N \}} \log n(z) & \leq & N \log n_* +
\sum_{k=k_*}^{k_N} m_k \log b_{k+1} \\[.2cm]
& \leq & N \left(\log n_* + \frac{2e}{\upsilon}
\sum_{k=k_*}^{\infty} \frac{1}{k^{1+\varepsilon}}\right).
\end{eqnarray*}
Applying (\ref{Mak3}) and the latter estimate in (\ref{Mak2}) we
obtain (\ref{J2}) also in this case. $\square$
\begin{remark}
  \label{NU1rk}
Our choice of $\phi$ made in (\ref{march}) was predetermined by the
condition (\ref{Mak3}), which we used to estimate the first summand
in $\{\cdot\}$ in  (\ref{Mak2}), as well as by the following one
\begin{equation} \label{Mak10}
\sum_{k=k_*}^{\infty} \frac{\log b_{k+1}}{\phi(b_k)} <\infty ,
\end{equation}
which was employed for estimating the second summand in
(\ref{Mak2}), for a concrete choice of the sequence $\{b_k\}_{k\geq
k_*}$ made therein. In principle, any $\phi$ obeying such two
conditions (for some choice of $\{b_k\}_{k\geq k_*}$) can be used.
For $b_k = k$, $k\geq k_* = n_*+1$, one can take $\phi (b) =
b^{1+\varepsilon}$ for some $\varepsilon >0$, which obviously obeys
(\ref{Mak3}) and (\ref{Mak10}) but imposes a stronger repulsion, see
(\ref{I}). Our choice (\ref{march}) seems to be optimal.
\end{remark}

{\sc Proof on Lemma \ref{Klm}.} In view of  (\ref{KK}), we have that
$\rho(x,y) \geq 2$ for any $x,y\in{\sf V}_*^c$; hence, for two
adjacent vertices, at least one should be in ${\sf V}_*$. Taking
this into account  by (\ref{Pq}) and the triangle inequality we
derive
\begin{eqnarray*}
\mathit{\Theta} (\alpha , \theta) & = & \sum_{x} [n(x)]^{\theta}
\left(\sum_{y\sim x} [n(y)]^{\theta}\right) \exp[-\alpha \rho
(o,x)]\\[.2cm] & \leq & n_*^\theta e^\alpha \sum_{y} [n(y)]^{1+\theta} \exp[-\alpha \rho
(o,y)]\nonumber\\[.2cm] & + & n_*^\theta \sum_{x}[n(x)]^{1+\theta} \exp[-\alpha \rho
(o,x)],\nonumber
\end{eqnarray*}
which  yields (\ref{J}), see (\ref{J0}).
 $\square$

\section*{Acknowledgements}

The authors are grateful to Philippe Blanchard and Michael R\"ockner
for valuable discussions and encouragement.

\end{document}